\documentclass[12pt]{article}

\usepackage[utf8]{inputenc}
\pdfoutput=1
\usepackage[sc,osf]{mathpazo}
\usepackage[margin=1in,letterpaper]{geometry}
\usepackage[round]{natbib}
\usepackage{tikz}
\usepackage{microtype}

\usepackage{hyperref}

\makeatletter
\newlength\aftertitskip     \newlength\beforetitskip
\newlength\interauthorskip  \newlength\aftermaketitskip

\def\maketitle{\par
 \begingroup
   \def\thefootnote{\fnsymbol{footnote}}
   \def\@makefnmark{\hbox to 4pt{$^{\@thefnmark}$\hss}}
   \@maketitle \@thanks
 \endgroup
\setcounter{footnote}{0}
 \let\maketitle\relax \let\@maketitle\relax
 \gdef\@thanks{}\gdef\@author{}\gdef\@title{}\let\thanks\relax}

\def\@startauthor{\noindent \normalsize\bf}
\def\@endauthor{}
\def\@starteditor{\noindent \small {\bf Editor:~}}
\def\@endeditor{\normalsize}
\def\@maketitle{\vbox{\hsize\textwidth
 \linewidth\hsize \vskip \beforetitskip
 {\begin{center} \LARGE\@title \par \end{center}} \vskip \aftertitskip
 {\def\and{\unskip\enspace{\rm and}\enspace}%
  \def\addr{\small\it}%
  \def\email{\hfill\small\tt}%
  \def\name{\normalsize\bf}%
  \def\AND{\@endauthor\rm\hss \vskip \interauthorskip \@startauthor}
  \@startauthor \@author \@endauthor}
}}

\makeatother

\pdfoutput=1                    

\usepackage{amsmath,amsthm}
\usepackage{graphicx}
\usepackage{color}
\usepackage{amssymb,amsfonts,amsxtra,mathrsfs,bm}
\usepackage{multicol}
\usepackage{multirow}
\usepackage{paralist}
\usepackage{xspace}
\usepackage{algorithm}
\usepackage{algpseudocode}
\definecolor{darkblue}{rgb}{0.0,0.0,0.65}
\definecolor{darkred}{rgb}{0.68,0.05,0.0}
\definecolor{darkgreen}{rgb}{0.0,0.29,0.29}
\definecolor{darkpurple}{rgb}{0.47,0.09,0.29}

\usepackage{hyperref}
\hypersetup{colorlinks = true,citecolor  = darkblue,linkcolor  = darkred,urlcolor   = darkblue}

\newtheorem{theorem}{Theorem}
\newtheorem{defn}[theorem]{Definition}
\newtheorem{lem}[theorem]{Lemma}
\newtheorem{prop}[theorem]{Proposition}
\newtheorem{prob}[theorem]{Problem}
\newtheorem{example}[theorem]{Example}
\newtheorem{rmk}[theorem]{Remark}
\newtheorem{cor}[theorem]{Corollary}

\usepackage{bbm}
\usepackage{nicefrac}

\newcommand{\Xc}{\mathcal{X}}

\newcommand{\Ac}{\mathcal{A}}

\newcommand{\BL}{{\rm BL}}

\newcommand{\nlsum}{\sum\nolimits}

\newcommand{\nlprod}{\prod\nolimits}
\newcommand{\reals}{\mathbb{R}}

\newcommand{\R}{\mathbb{R}}
\newcommand{\posdef}{\mathbb{P}}
\newcommand{\pd}{\posdef}
\renewcommand{\H}{\mathbb{H}}

\newcommand{\vw}{\bm{w}}

\newcommand{\ip}[2]{\langle {#1},\, {#2} \rangle}

\newcommand{\norm}[1]{\|{#1}\|}

\DeclareMathOperator{\trace}{tr}

\renewcommand{\H}{\mathbb{H}}

\usepackage{booktabs} 
\usepackage{makecell} 

\numberwithin{equation}{section}

\title{Computing Brascamp-Lieb Constants \\ through the lens of Thompson Geometry}
\author{\name Melanie Weber \email{mweber@seas.harvard.edu}\\
  \addr{Harvard University}\\
  \name Suvrit Sra \email{suvrit@mit.edu}\\
  \addr{MIT, Laboratory for Information and Decision Systems}
}

\begin{document}
\maketitle

\begin{abstract}
This paper studies algorithms for efficiently computing \emph{Brascamp--Lieb} constants, a task that has recently received much interest. In particular, we reduce the computation to a nonlinear matrix-valued iteration, whose convergence we analyze through the lens of fixed-point methods under the well-known Thompson metric. This approach permits us to obtain (weakly) polynomial time guarantees, and it offers an efficient and transparent alternative to previous state-of-the-art approaches based on Riemannian optimization and geodesic convexity.
\end{abstract}

\section{Introduction}
In a seminal paper,~\citep{BL1} introduced a class of inequalities (\emph{Brascamp--Lieb inequalities}) that generalizes many well-known inequalities, including H{\"o}lder's inequality, the sharp Young convolution inequality, and the Loomis--Whitney inequality.  As such,  the class of Brascamp-Lieb inequalities provides a valuable tool of use to many areas of mathematics~\citep{tao-paper}.   Brascamp--Lieb inequalities have also been applied in Machine Learning~\citep{dvir2016rank,pmlr-v30-Hardt13} and Information Theory~\citep{carlen2009subadditivity,liu2016smoothing}. 

Each Brascamp--Lieb (BL) inequality is characterized by a \emph{BL-datum} $(\Ac,w)$, where $\Ac = \big( A_1, \dots, A_m \big)$ is a tuple of linear transformations and  $w=\lbrace w_j \rbrace$ is a set of exponents.  Of central importance is the feasibility of a BL-datum, which is characterized by the following quantity:
\begin{defn}[Brascamp--Lieb Constant]
Let $\Ac := \big( A_1, \dots, A_m \big)$ be a tuple of surjective linear transformations $A_j: \reals^d \rightarrow \reals^{d_j}$ and $\vw :=(w_1,\ldots,w_m)$ a real vector with positive entries.  For each BL-datum $(\Ac,\vw)$ there exists a constant $C(\Ac,\vw)$ (finite or infinite) such that for any set of real-valued, non-negative, Lebesgue-measurable functions $f_j: \R^{d_j} \rightarrow \R, \; j \in [m]$, the following (\emph{Brascamp--Lieb}) inequality holds
\begin{equation}\label{eq:BL-inequ}
\int_{x \in \R^d} \Bigl( \nlprod_{j \in [m]} f_j (A_j x)^{w_j} \Bigr) dx 
\leq C(\Ac,\vw) \nlprod_{j \in [m]} \Bigl( \int_{x \in \R^{d_j}} f_j (x) dx	\Bigr)^{w_j} \; .
\end{equation}
The smallest such constant that satisfies~\eqref{eq:BL-inequ} for a BL-datum $(\Ac,w)$ and for any choice of $(f_1, \dots, f_m)$, is called the \emph{Brascamp--Lieb constant} $\BL(\Ac,\vw)$.  The datum $(\Ac,\vw)$ is said to be \emph{feasible} if $\BL(\Ac,\vw) < \infty$, and \emph{infeasible} otherwise.
\end{defn}

A necessary and sufficient condition for feasibility of a Brascamp--Lieb data is given in~\citep[Theorem~1.13]{tao-paper} (see Theorem~\ref{thm:feasible}).
Our aim is to characterize Brascamp--Lieb constants of feasible data; as we will see in section~\ref{sec:BL-basics}, we have
\begin{align}\label{eq:BL-const}
\sup_{Z \in \; \prod_{j \in [m]} \mathbb{P}_{d_j}} \biggl[\BL(\Ac,\vw; Z) := \biggl(	\frac{\prod_{j \in [m]} \det(Z_j)^{w_j}}{\det \bigl( \sum_{j \in [m]} w_j A_j^* Z_j A_j\bigr)}	\biggr)^{1/2} \biggr] \; ,
\end{align}
where the supremum is taken over a product of positive definite matrices ($\pd_d$).
Then, the key question is whether there exists a solution to the following problem:
\begin{prob}\label{prob:BL}
Given a feasible datum $(\Ac,\vw)$, can we compute an $\epsilon$-approximation to $\BL(\Ac,\vw)$, in time that is polynomial in the number of bits required to represent the datum? That is, for any $\epsilon>0$, can we compute a constant $C$, such that 
\begin{align*}
\BL(\Ac,\vw) \leq C \leq (1+\epsilon) \BL(\Ac,\vw),
\end{align*}
in time that is polynomial in the bit length of the datum $(\Ac, \vw)$ and $\log(1/\epsilon)$.
\end{prob}

\noindent We will see below (section~\ref{sec:BL-basics}) that the computation of the supremum in~\eqref{eq:BL-const} is equivalent to solving the following optimization task~\citep[Definition~7]{vishnoi}:
\begin{prop}
\label{prop:BL-min}
    Let $(\Ac,\vw)$ denote a simple BL-datum and let $\vw=(w_1, \dots, w_n)$ with $w_j \in (0,1)$ and $\sum_{j \in [m]} w_j =1$. Then Problem~\ref{prob:BL} is equivalent to solving the following optimization problem:
    \begin{equation}\label{eq:F}
        \min_{X \in \mathbb{P}_d} F(X) := \nlsum_{j \in [m]} w_j \log \det \big(	A_j^\ast X A_j	\big) - \log \det (X) \; .
    \end{equation}
\end{prop}

 A seminal result by~\citep[Theorem 3.2]{lieb} ensures that the optimum is positive definite, which allows for solving~\eqref{eq:F} over the positive definite matrices ($\pd_d$).
Importantly, problem~\eqref{eq:F} turns out to be a geodesically convex optimization problem on the manifold  of symmetric positive definite matrices (under a suitable Riemannian metric~\citep{vishnoi}).

Given that $F$ in~\eqref{eq:F} is geodesically convex, it is natural to consider Riemannian optimization approaches. For using first-order methods such as Riemannian Gradient Descent~\citep{absil2009optimization,zhang}, geodesic convexity  proves key to ensuring sublinear rates, such as $O \big(\frac{1}{\epsilon} \big)$, on \emph{function suboptimality} $F(X_k)-F(X^*) \le \epsilon$. We will see below that the optimum lies in a compact convex set; hence \eqref{eq:F} could also be viewed as a constrained optimization task with a geodesic ball constraint. We could exploit this constraint by using first-order constrained optimization methods such as Projected Riemannian Gradient Descent~\citep{absil2009optimization,zhang} or Riemannian Frank--Wolfe~\citep{weber2022riemannian,weber2021projection}, both of which again achieve only sublinear convergence rates on function suboptimality.  Furthermore,  for a different geodesically convex formulation of~\eqref{eq:BL-const} due to~\citep{garg2018algorithmic}, it is known that second-order methods can attain a faster $O \bigl({\rm poly} \bigl( \log \frac{1}{\epsilon}\bigr)\bigr)$ convergence rate~\citep{allen-zhu}. Under additional spectral assumptions, even linear convergence at rate $O \bigl( \log \frac{1}{\epsilon}\bigr)$ is possible~\citep{kwok}.  

\subsection{Our approach} We take a different route in this paper; our idea is to analyze a \emph{fixed-point iteration} that solves~\eqref{eq:F}. In particular, we study Picard iterations obtained by variations of the nonlinear map $G: \mathcal{X} \rightarrow \mathcal{X}$ (where $\Xc$ is a suitable subset of $\pd_d$), defined as
\begin{equation}\label{eq:map}
G: \quad X \mapsto \Big(\nlsum_{j \in [m]} w_j A_j \big(	A_j^\ast X A_j 	\big)^{-1} A_j^\ast \Big)^{-1} \; .
\end{equation}

The simplicity of the map~\eqref{eq:map} makes a fixed-point approach particularly attractive, since it avoids expensive Riemannian operations such as exponential maps and parallel transports, which are required by standard Riemannian optimization approaches. Empirically, iterating map~\eqref{eq:map} exhibits linear convergence -- see Figure~\ref{fig:BL} for numerical experiments with varying input sizes. The faster convergence in comparison with Riemannian gradient descent illustrates the attractiveness of our approach.
\begin{figure}[t]
    \centering
    \includegraphics[scale=0.15]{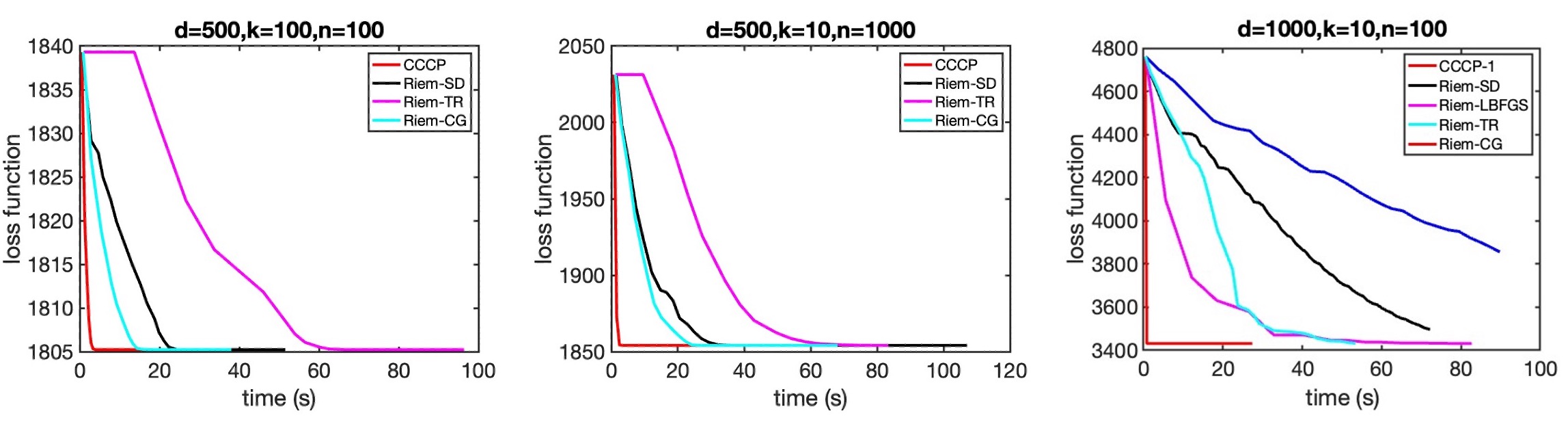}
    \caption{Empirical performance of the map $G$ (Eq.~\eqref{eq:map}) in comparison with first-order Riemannian Optimization methods (using the Manopt implementation~\citep{boumal_manopt_2014}): Riemannian Steepest Descent (Riem-SD), Riemannian Trustregions (Riem-TR) and Riemannian Conjugate Gradient (Riem-CG). The Brascamp--Lieb datum consists of a set of $n$ linear transformations $A_j: \R^d \rightarrow \R^k$ and an $n$-dimensional vector that characterizes the exponents. The reported value of the loss function is the value of the objective~\ref{eq:F}, i.e., the value attained by the BL constant.}
    \label{fig:BL}
\end{figure}

We will present a detailed non-asymptotic convergence analysis of a fixed-point approach based on iterating~\eqref{eq:map} for simple input data (see Def.~\ref{def:simple}). The key insight of our work is to analyze the nonlinear map $G$ through the lens of nonlinear Perron--Frobenius theory, which  views $\pd_d$ as a metric space endowed with the Thompson part metric. 
Our initial convergence analysis hinges on establishing \emph{non-expansivity} of $G$ with respect to the Thompson part metric, as well as \emph{asymptotic regularity} of its ``average map'' $G_t$, i.e.,
$\delta_{\rm T} \big( G_t(X_{k+1}), G_t(X_k) \big) \rightarrow 0$ ($G_t$ is defined in Eq.~\ref{eq:G-avg}; a formal definition of asymptotic regularity is given in Definition~\ref{def:asymp-reg}). Here, $\delta_{\rm T}$ denotes the Thompson metric, which is defined in Definition~\ref{def:thompson} below.
More precisely, we make use of the following result, which adapts a classical result on the convergence of non-expansive maps in Banach spaces (see, e.g.,~\citep[Theorem 1.2]{baillon_asymptotic_1978}) to Thompson geometry: 
\begin{theorem}\label{thm:s-h}
Let $T$ be a non-expansive and asymptotically regular nonlinear map $T: \Xc \rightarrow \Xc$ ($\Xc \subset \pd_d$), and let 
the set of its fixed points ${\rm Fix}(T) := \lbrace X^\ast \in \mathbb{P}_d: \; X^\ast = T(X^\ast) \rbrace$ be non-empty. Then the sequence $(X_k)_{k \in \mathbb{N}}$ defined by $X_{k+1}=T(X_k)$ ($X_0 \in \Xc$) converges to a point in ${\rm Fix}(T)$.
\end{theorem}

\noindent Our plan is as follows: We begin by showing that $G$ has a fixed point in some compact and convex set. 
We then establish the non-expansivity of $G$, which implies that its ``average map'' $G_t$ is non-expansive and has the same fixed point set. It further implies that $G_t$ is asymptotically regular, which ensures asymptotic convergence (following Theorem~\ref{thm:s-h}) to a fixed point of $G$. Subsequently, we refine our convergence analysis to seek non-asymptotic guarantees that ensure an $O \big(\log \frac{1}{\epsilon} \big)$ convergence rate. To this end, we show that a regularized variant of $G$ is \emph{strictly} contractive with respect to the Thompson part metric. Finally, we describe how this analysis implies a convergence rate for solving the original problem~\eqref{eq:F}, completing the analysis.

\subsection{Summary of contributions}
The main contribution of this work is a novel geometric lens on Brascamp-Lieb constants, which relies on a Finsler geometry on the manifold of positive definite matrices, induced by the Thompson part metric.
Specifically,  our contributions are as follows:
\begin{enumerate}
  \setlength{\itemsep}{0pt}
\item We propose fixed-point approaches for computing Brascamp--Lieb constants of feasible data that is obtained by analyzing the first-order optimality conditions of the geodesically convex problem~\eqref{eq:F}.
\item We analyze the proposed approach and derive non-asymptotic guarantees on its convergence for simple input data.  Specifically, we study a regularized variant of our fixed-point approach that solves the original problem~\eqref{eq:F} to $\epsilon$-accuracy with $O \Big( \log \big( \frac{1}{\epsilon},\frac{\Delta}{\delta^{3/2}} ,\sqrt{d}, \frac{1}{R}\big) \Big)$ iterations. 
Here, $d, \delta, \Delta, R$\footnote{Note that the constants $\delta, \Delta, R$ are not necessarily independent of each other. In particular, to be more precise, we should write $R(\delta)$, to indicate that $R$ may depend on $\delta$. However, to keep notation simple, we write just $R$.} are constants, which depend on the input datum $(\Ac,\vw)$ only.
\end{enumerate}
Our analysis leverages the Thompson part metric on the manifold of positive definite matrices to model convergence of the fixed-point iteration. To our knowledge, this is the first work that analyzes the computation of Brascamp--Lieb constants via Thompson geometry. 
We note that a similar Finslerian lens can be employed to understand other Picard iterations arising from problem~\eqref{eq:F}. We discuss this observation in more detail below.
Our proof techniques build on ideas from nonlinear Perron-Frobenius theory~\citep{lemmens_nussbaum_2012}, and may be of independent interest for related problems. 
Moreover, the simplicity of our approach and its ease of implementation make it attractive for applications. However, a key limitation of our approach is the difficulty of choosing a good initialization in practice, which relies on an efficient construction of the parameters $\delta, \Delta$. At present, we do not know of such a construction. We will discuss this limitation in detail in sections~\ref{sec:delta} and~\ref{sec:discussion}. 

\subsection{Outline}
The paper is structured as follows: In section~\ref{sec:background} we introduce basic background and notation, including Thompson geometry on the space of positive definite matrices and the class of Brascamp--Lieb inequalities. In section~\ref{sec:overview}, we provide an overview of the paper's main results and state the main theorems. In section~\ref{sec:related}, we review related literature and existing results on computing Brascamp--Lieb constants. Section~\ref{sec:proofs} provides formal proofs of the paper's findings. In section~\ref{sec:other-picard}, we provide a discussion on alternative Picard iterations that could be use to compute Brascamp--Lieb constants. We conclude with a general discussion of the paper's results and limitations, as well as avenues for future investigation, in section~\ref{sec:discussion}.

\section{Background and Notation}
\label{sec:background}
\subsection{Thompson geometry on positive definite matrices} 
We start by recalling some background and notation. Throughout, we consider optimization tasks on the manifold of positive definite matrices, i.e., 
\begin{align*}
\pd_d := \lbrace X \in \H_d: \; X \succ 0	\rbrace \; ,
\end{align*}
where $\H_d \subset \R^{d \times d}$ denotes the set of order-$d$ real symmetric matrices and ``$\succ$'' denotes the Löwner partial order on $\H_d$, where ``$X \succ Y$" (``$X \succeq Y$") implies that $X$ has only positive (non-negative) eigenvalues.  
With a slight abuse of notation, we use $\pd_d$ to denote \emph{symmetric} positive definite matrices.  Recall the \emph{Schatten $p$-norm}, which on $\mathbb{P}_d$ becomes $\norm{X}_p := \big( \trace X^p \big)^{1/p}$.  Important cases are the \emph{trace norm} $\norm{X}_1 = \trace(X)$ ($p=1$) and the \emph{Frobenius norm} $\norm{X}_{\rm F} =\big( \trace X^2\big)^{1/2}$ ($p=2$).  We further use the usual operator norm $\norm{\cdot}$; 
note that $\norm{X} \leq \norm{X}_{\rm F} \leq \sqrt{d} \norm{X}$.

Most Riemannian optimization literature that studies problems on the manifold $\mathbb{P}_d$ considers the Riemannian geometry induced by the following metric:
\begin{equation}
\ip{A}{B}_X := \trace   (X^{-1} A X^{-1} B), \quad X \in \pd_d, A,B \in T_X(\pd_d)= \mathbb{H}_d \; ,
\end{equation}
where the tangent space $\mathbb{H}_d$ is the space of symmetric matrices.  Equipped with this metric, $\pd_d$ is a Cartan--Hadamard manifold. Its (global) non-positive sectional curvature gives rise to many desirable properties, including geodesic convexity. 

While we use the Riemannian lens above for exploiting the geodesic convexity of~\eqref{eq:F}, our algorithm design and analysis views $\pd_d$ through the lens of nonlinear Perron--Frobenius theory~\citep{lemmens_nussbaum_2012} instead. 
Here, we view $\pd_d$ as a convex cone  that is endowed with the Thompson part metric, which induces a Finsler geometry~\citep{Finsler1918} on $\pd_d$. Our analysis does not require any specific Finslerian properties; hence, we recall only basic definitions and some key properties below. However, we briefly remark on a key difference between the Riemannian and Finslerian geometries: While on Riemannian manifolds the inner product defines the norm on the tangent spaces, norms need not be induced by inner products on Finslerian manifolds. Specifically, tangent spaces in the Thompson geometry resemble L-infinity spaces instead of the usual L-2 structure in the Riemannian setting.
\begin{defn}[Thompson metric]\label{def:thompson}
The Thompson metric $\delta_{\rm T}$ is defined as
\begin{equation}
\delta_{\rm T}(X,Y) := \log \max \lbrace M(X/Y),  M(Y/X)	\rbrace,
\end{equation}
where $M(X/Y) := \inf \lbrace	\lambda > 0 : X \preceq \lambda Y	\rbrace$.
\end{defn}
\begin{lem}[Properties of the Thompson metric]\label{lem:prop-dT}
We assume $X,Y \in \pd_d$.
\begin{enumerate}
\item $\delta_{\rm T} (X^{-1},Y^{-1}) = \delta_{\rm T} (X,Y)$;
\item $\delta_{\rm T} (B^\ast X B, B^\ast Y B) = \delta_{\rm T} (X,Y)$ for $B \in GL_n(\mathbb{C})$;
\item $\delta_{\rm T} (A^\ast X A, A^\ast Y A) \leq \delta_{\rm T} (X,Y) $ for $A \in \mathbb{C}^{d \times r}$, having full rank;
\item $\delta_{\rm T} (X^t, Y^t) \leq \vert t \vert \delta_{\rm T} (X,Y)$ for $t \in [-1,1]$;
\item $\delta_{\rm T} (A+B, C+D) \leq \max \lbrace	 \delta_{\rm T}(A,C), \delta_{\rm T}(B,D)	\rbrace$, which can be recursively generalized to sums with $m$ terms, i.e.
\begin{align*}
\delta_{\rm T} \Big( \nlsum_{j \in [m]} X_j, \nlsum_{j \in [m]} Y_j	\Big) \leq \max_{1 \leq j \leq m} \delta_{\rm T} (X_j, Y_j) \; ;
\end{align*}
\item $\delta_{\rm T} (X+A, Y+A) \leq \frac{\alpha}{\alpha + \beta} \delta_{\rm T} (X,Y)$, where $A \succeq 0$, $\alpha = \max \{\norm{X},\norm{Y}\}$, and $\beta = \lambda_{\min}(A)$.
\end{enumerate}
\end{lem}
\noindent For a proof of Lemma~\ref{lem:prop-dT} we refer the reader to~\citep{sra-hosseini}.\\

\noindent Furthermore, we will use the following notions of non-expansitivity and contractivity with respect to the Thompson metric.
\begin{defn}\label{def:contract}
We say that a map $G: \Xc \rightarrow \Xc$ ($\Xc \subset \pd_d$) is
\begin{enumerate}
\item \emph{non-expansive} if $\delta_{\rm T}(G(X),G(Y)) \leq \delta_{\rm T} (X,Y)$ for all $X,Y \in \mathcal{X}$;
\item \emph{contractive} if $\delta_{\rm T}(G(X),G(Y)) < \delta_{\rm T}(X,Y)$ for all $X,Y \in \mathcal{X}$.
\end{enumerate}
\end{defn}
\noindent Lastly,  recall that the \emph{Kronecker product} $X \otimes Y$ of matrices $X \in \R^{m \times n}, Y \in \R^{p \times q}$ is given as $(X \otimes Y)_{\alpha \beta} = a_{ij} b_{kl}$, where $\alpha = p(i-1)+k$, $\beta = q(j-1)+l$.  The Kronecker product is bilinear and associative.

\subsection{Brascamp--Lieb inequalities}
\subsubsection{Basic notation}\label{sec:BL-basics}
The goal of the present paper is the computation of Brascamp--Lieb constants for feasible data.  The following seminal result by~\citep{lieb} allows for phrasing this problem as an optimization task on $\pd_d$.
\begin{theorem}[\citep{lieb}, Theorem 3.2]
Let $(\Ac,\vw)$ denote a BL-datum, where $A_j: \reals^d \rightarrow \reals^{d_j}$ for all $A_j \in \Ac$ ($1\le j\le m$).  Consider the function
\begin{equation}\label{eq:BL-lieb}
\BL(\Ac,\vw; Z) := \left(	\frac{\prod_{j \in [m]} \det(Z_j)^{w_j}}{\det \bigl( \sum_{j \in [m]} w_j A_j^* Z_j A_j\bigr)}	\right)^{1/2}, 
\end{equation}
where $\bm{Z} := (Z_1,\dots, Z_m) \in \; \prod_{j \in [m]} \pd_{d_j}$.  The Brascamp--Lieb constant for $(\Ac,\vw)$ can be computed as
\begin{equation}\label{eq:BL-sup}
\sup_{Z \in \; \prod_{j \in [m]} \pd_{d_j}} \BL \big(\Ac,\vw; Z \big) \; .
\end{equation}
\end{theorem}
\noindent \citep{tao-paper} characterized the feasibility of $(\Ac,\vw)$ as follows:
\begin{theorem}[\citep{tao-paper},Theorem~1.13]\label{thm:feasible}
$(\Ac,\vw)$ is feasible if and only if
\begin{enumerate}
\item $d=\sum_{j \in [m]} w_j d_j$,
\item $\dim(H) \leq \sum_{j \in [m]} w_j \dim(A_j H)$ for any subspace $H \subseteq \R^d$.
\end{enumerate}
\end{theorem}
\noindent Notice that, in particular, Theorem~\ref{thm:feasible}(2) implies that $\bigcap_{j \in [m]} {\rm ker} A_j=\{0\} $ for $j \in [m]$. This result renders the feasibility of Brascamp--Lieb data into a constrained optimization problem: It induces linear constraints on $\vw$ as $H$ varies over different subspaces.  Importantly, since $\dim(A_j H) \leq d$,  the set of constraints is finite. 
Therefore, the feasible set is a convex polytope $\Xc_{(\Ac,\vw)}$ (also known as \emph{Brascamp--Lieb polytope}).  If $\vw$ lies on the boundary of $\Xc_{(\Ac,\vw)}$, then there exists a non-trivial subspace $H$ for which the inequality constraints in Theorem~\ref{thm:feasible} are tight. 
\begin{defn}[Simple Brascamp--Lieb data, \cite{tao-paper}]
\label{def:simple}
Let $(\Ac,\vw)$ feasible. A subspace $H \subseteq \R^d$ is \emph{critical} if
\begin{equation}
\dim(H) = \sum_{j \in [m]} w_j \dim(A_j H) \; .
\end{equation}
We say that $(\Ac,\vw)$ is \emph{simple} if there exists no non-trivial, critical subspace.
\end{defn}
Geometrically, simple Brascamp--Lieb data lie in the interior of $\Xc_{(\Ac,\vw)}$.  Importantly, Problem~\eqref{eq:BL-const} has a unique global solution for simple data.
\begin{theorem}\label{thm:maximizer}
Let $(\Ac,\vw)$ as before and $X := \sum_{j \in [m]} w_j A_j^* Z_j A_j$. Then, the following statements are equivalent:
\begin{enumerate}
\item $Z$ is a global maximizer for $\BL(\Ac,\vw,Z)$;
\item $Z$ is a local maximizer for $\BL(\Ac,\vw,Z)$;
\item $X$ is invertible and $Z_j^{-1} = A_j X^{-1} A_j^*$.
\end{enumerate}
\end{theorem}
\noindent Theorem~\ref{thm:maximizer} implies that the feasibility of a BL datum can be characterized by maximizing $\BL(\Ac,\vw,Z)$.  Moreover,  Theorem~\ref{thm:maximizer}(3) allows for rewriting~\eqref{eq:BL-sup} with respect to $X$, which results in the formulation that we will analyze in this work (Proposition~\ref{prop:BL-min}).

\subsubsection{Examples}
Many well-known inequalities in mathematics can be written as Brascamp--Lieb inequalities. In this section, we recall a few key examples.  A more comprehensive list of examples can be found in~\citep{tao-paper}.
\begin{example}[H{\"o}lder]
Consider $A_j = {\rm id}_{\R^d}$ for all $j$, i.e., $\Ac = ({\rm id}_{\R^d} )_{1 \leq j \leq m}$. Then the multilinear \emph{H{\"o}lder inequality} asserts that $\BL(\Ac,\vw)=1$ if $\sum_{j \in [m]} w_j = 1$ and $\BL(\Ac,\vw)=+\infty$ otherwise.
\end{example}
\begin{example}[Sharp Young inequality]
Consider $A_j: \R^d \times \R^d \rightarrow \R^d$ for $j=1,2,3$ with
\begin{equation*}
\begin{cases}
A_1(x,y) := x \\
A_2(x,y) := y \\
A_3(x,y) := x-y.
\end{cases}
\end{equation*}
Then the \emph{Sharp Young inequality} asserts that
\begin{equation*}
\BL(\Ac,\vw) = \Big( \prod_{j=1}^3 \frac{(1-w_j)^{1-w_j}}{w_j^{w_j}}	\Big)^{d/2} \; ,
\end{equation*}
if $\sum_{j=1}^3 w_j =2$ and $\BL(\Ac,\vw)=+\infty$ otherwise.
\end{example}
\citep{barthe} introduced a class of \emph{Reverse Brascamp--Lieb inequalities} that generalize several important inequalities that cannot be encoded with the original Brascamp--Lieb framework.  A key example is the \emph{Gaussian Brunn--Minkowski inequality}~\citep{barthe2008gaussian}.  Importantly, since Reverse Brascamp--Lieb inequalities have the same data as the standard Brascamp--Lieb inequalities,  there exists an intrinsic relation between their constants: The optimal constants of standard and Reverse Brascamp--Lieb inequalities multiply to 1 if the datum is feasible. Infeasible data have Brascamp--Lieb constants of $+\infty$ and Reverse Brascamp--Lieb constants of $0$~\citep{barthe}. This connection  ensures that the algorithmic results presented here and in the related work translate from standard to Reverse Brascamp--Lieb inequalities.

\section{Overview of our results}\label{sec:overview}
We start with a derivation of the map~\eqref{eq:map} with the goal of establishing that its Picard iteration solves problem~\eqref{eq:F}. We start by computing the gradient of $F$:
\begin{align}
\nabla F(X) = \nlsum_{j \in [m]} w_j A_j (	A_j^\ast X A_j 	)^{-1} A_j^\ast - X^{-1} \; .
\end{align}
Since the constraint set $\pd_d$ is open, $\nabla F(X)=0$ is necessary for $X$ to be a local minimum. Importantly, since $F$ is geodesically convex\footnote{For a comprehensive introduction to geodesically convex optimization, see~\citep{udriste1994convex}.}, this local minimum must be global. Thus, to solve~\eqref{eq:F}, it suffices to solve the equation
$\nabla F(X)=0$, or equivalently $\nlsum_{j \in [m]} w_j A_j (	A_j^\ast X A_j )^{-1} A_j^\ast - X^{-1} = 0$,
which can be rewritten as the nonlinear matrix equation
\begin{equation}
    \label{eq:pre1}
    X = \bigl( \nlsum_{j} w_j A_j (	A_j^\ast X A_j 	)^{-1} A_j^\ast \bigr)^{-1}.
\end{equation}
Once written as~\eqref{eq:pre1}, it is natural to consider the following iterative map:
\begin{equation*}
G: \quad X \mapsto \Big(\nlsum_{j \in [m]} w_j A_j \big(	A_j^\ast X A_j 	\big)^{-1} A_j^\ast \Big)^{-1} \; .
\end{equation*}
Remarkably, this na\"ive consideration turns out to be a meaningful algorithm. That is, we can show that for simple input data, the sequence of iterates $(X_k)_{k>0}$ produced by iterating~\eqref{eq:map} has a fixed point under the usual assumptions.  A crucial ingredient of our analysis is the observation that the solution of Problem~\eqref{eq:F} is contained in a compact convex set of the form
\begin{equation}\label{eq:def-D}
D_{\delta, \Delta} := \big\{X \in \pd_d: \; \delta I_d	\preceq X  \preceq \Delta I_d \big\} \; ,
\end{equation}
where $\delta, \Delta >0$ are parameters that depend only on the input datum $(\Ac, \vw)$. 
The upper and lower bounds on the solution are implied by~\citep[Proposition~5.2]{tao-paper}.

Our convergence analysis will utilize the ``average map'' $G_t$ of $G$, which is obtained by setting
\begin{equation}\label{eq:G-avg}
    G_t := (1-t)I + tG \; , \qquad t \in (0,1) \; .
\end{equation}
\noindent Importantly, we will establish that $G_t$ shares the fixed point set of $G$. Formally, we will show the following convergence result:
\begin{theorem}\label{thm:conv-G1}
There exist $\delta, \Delta \in (0,\infty)$ that depend only on the BL data $(\Ac,\vw)$ such that $\delta < \Delta$, and for every $X_0 \in D_{\delta, \Delta}$ the sequence $(X_k)_{k \in \mathbb{N}}$ generated by iterating $G_t$, i.e., $X_{k+1}=G_t(X_k)$, converges to a solution of~\eqref{eq:F}.
\end{theorem}
\noindent The proof of this key result relies on Theorem~\ref{thm:s-h}. We have to show the three conditions of the theorem: (1) that $G_t$ has a fixed point in $D_{\delta,\Delta}$, (2) that $G_t$ is non-expansive on $D_{\delta,\Delta}$, and (3) that $G_t$ is asymptotically regular.  
\noindent Proofs for all three conditions will be given in Section~\ref{sec:convergence-G}. 
The existence of a suitable set $D_{\delta, \Delta}$ and its dependency on the input datum $(\Ac,\vw)$ via $\delta$ and $\Delta$ will be discussed in Section~\ref{sec:delta}.

However, to establish stronger, \emph{non-asymptotic} convergence guarantees, \emph{strict} contractivity (see Definition~\ref{def:contract}(2)) of the 
iterative map is typically required. Unfortunately, $G$ (and $G_t$) only fulfills the weaker notion of non-expansivity (Definition~\ref{def:contract}(1)). Therefore, we cannot hope to develop non-asymptotic convergence guarantees with the tools proposed so far. To mitigate this problem, we introduce a regularization of $G$, which instead of minimizing $F$ directly, minimizes a perturbed objective. Consider thus the regularized problem
\begin{equation}\label{eq:F-regul}
\min_{X \in \pd_d} \; F_{\mu} (X) := F(X) + \mu \trace (X) \; ,
\end{equation}
for some $\mu > 0$. By analyzing the gradient of $F_{\mu}$, we can derive an iterative procedure for solving Problem~\eqref{eq:F-regul} as
\begin{align*}
\nabla F_\mu (X) &= \nabla F(X) + \mu I = 0, \\
X^{-1} &= \mu I + \nlsum_{j \in [m]} w_j A_j (A_j^\ast X A_j)^{-1} A_j^\ast \\
X &= \bigl(\mu I + \nlsum_{j \in [m]} w_j A_j (A_j^\ast X A_j)^{-1} A_j^\ast \bigr)^{-1} \; ,
\end{align*}
using which we can define the regularized map
\begin{equation*}\label{eq:G2}
G_\mu: \quad X \mapsto \Big(\mu I + \nlsum_{j \in [m]} w_j A_j \big(	A_j^\ast X A_j 	\big)^{-1} A_j^\ast \Big)^{-1} \; .
\end{equation*}
By lemma~\ref{lem:G4-contract}, which is proved in section~\ref{sec:convergence-G-reg}, the regularized map $G_\mu$ is indeed contractive:
\begin{lem}\label{lem:G4-contract}
For any $\mu>0$, the map $G_\mu$ is contractive on $D_{\delta,\Delta}$, i.e.,
\begin{equation*}
    \delta_{\rm T}(G_\mu(X), G_\mu(Y))< \delta_{\rm T}(X,Y)
\end{equation*}
for all $X,Y \in D_{\delta,\Delta}$.
\end{lem}
\noindent Finally, we analyze the convergence rate of the proposed fixed-point approach.
The central result of the paper gives the following non-asymptotic convergence guarantee:
\begin{theorem}\label{thm:conv-non-asymp}
For a suitable choice of $\mu$, $G_{\mu}$ computes an $\epsilon$-approximate solution to problem~\eqref{eq:F}, i.e., a constant $C$, such that 
\begin{equation*}
    \BL(\Ac,\vw) \leq C \leq (1+\epsilon) \BL(\Ac,\vw) \; ,
\end{equation*}
with iteration complexity $O \Big( \log \big( \frac{1}{\epsilon},\frac{\Delta}{\delta^{3/2}} ,\sqrt{d}, \frac{1}{R}\big) \Big)$, where $d, \delta, \Delta, R$ depend on the input datum $(\Ac,\vw)$ only.
\end{theorem}
The proof of Theorem~\ref{thm:conv-non-asymp} relies on a convergence analysis of the regularized map $G_\mu$; 
a suitable choice of the parameter $\mu$ is discussed in section~\ref{sec:convergence-G-reg}. The latter relies on the smoothness of $F_\mu$, which is characterized by the following lemma:
\begin{lem}\label{lem:F-smooth}
${\rm Lip}(\nabla F_\mu) \leq \frac{2}{\delta}$, i.e., $\norm{\nabla F_\mu(X) - \nabla F_{\mu}(Y)} \leq \frac{2}{\delta} \norm{X-Y}$ for all $X, Y \in \pd_d$.
\end{lem}
The proofs of Theorem~\ref{thm:conv-non-asymp} and Lemma~\ref{lem:G4-contract} are given in Section~\ref{sec:convergence-G-reg}. We discuss the complexity of the algorithm in more detail in Section~\ref{sec:convergence-G-reg}.

A practical implementation of our fixed-point approach requires an initialization in the set $D_{\delta,\Delta}$ and in turn an explicit computation of the parameters $\delta$ and $\Delta$ for a given input datum $(\Ac,\vw)$.  We discuss the difficulty of this problem, as well as its connection to constrained subdeterminant maximization in Section~\ref{sec:delta}.

\begin{rmk}\normalfont
    The map $G$ resembles the alternate scaling algorithm for Brascamp-Lieb constants~\citep[Alg.~1]{garg2018algorithmic}. The resemblance of both approaches derives from an exploitation of the difference-of-convex structure of problem~\ref{eq:F} (see also~\citep{sra-weber-2022}). However, the Thompson geometry perspective employed in this work, utilizes fundamentally different tools than the Riemannian lens employed in the previous literature. To the best of our knowledge, the maps $G_t$ and $G_\mu$ have not been considered previously.
\end{rmk}

\section{Related work}\label{sec:related}
Computing Brascamp--Lieb constants has been an actively studied problem for many years. Following seminal work by Brascamp and Lieb~\citep{BL1,BL2,lieb} and Bennett et al.~\citep{tao-paper}, recent work has analyzed the problem through a geometric lens, exploiting geodesic convexity to utilize tools from Riemannian Optimization.

A growing body of literature has analyzed a formulation of the problem, which links to the more general \emph{operator scaling problem}~\citep{allen-zhu,garg,garg2018algorithmic,kwok,franks2018operator,burgisser2018efficient,burgisser2019towards}.  We briefly recall some of the key results in this line of work and comment on their relation to the current paper.

The \emph{operator scaling problem} introduced by Gurvits~\citep{gurvits} asks to compute, for a given operator $\Phi$, left and right scaling matrices, such that the scaled operator $\Phi'$ is doubly stochastic. Formally, let $\Ac = \big(A_1, \dots, A_m \big)$ be a tuple of matrices with $A_j \in \reals^{d \times d_j}$ that define a positive operator 
\begin{equation*}
\Phi_\Ac(X) = \sum_{j \in [m]} A_j X A_j^* \; .
\end{equation*}
Furthermore, consider the notion of \emph{capacity} for an operator $\Phi_\Ac$, defined as
\begin{align}\label{eq:cap}
{\rm cap} (\Phi_\Ac) := \inf_{X \succ 0, \det(X)=1} \det(\Phi_\Ac(X)) \; .
\end{align}
If an infimum $X^*$ can be attained in~\eqref{eq:cap}, then the scaling 
\begin{equation*}
    A_j' = \big(\Phi_\Ac(X^*)^{-1}	\big)^{1/2} A_j (X^*)^{1/2}
\end{equation*}
gives $\sum_{j \in [m]} A_j' (A_j')^* = 1$ and $\sum_{j \in [m]} (A_j')^* A_j' = 1$.  In consequence, the scaled operator 
\begin{align}
\Phi_\Ac'(X) = \sum_{j \in [m]} A_j' X (A_j')^*
\end{align}
is doubly stochastic as desired.  Therefore, minimizing the capacity provides a means to computing a suitable scaling.  A simple method for this task is Gurvits' algorithm~\citep{gurvits}, which can be viewed as a generalization of Sinkhorn's algorithm~\citep{sinkhorn}. The operator scaling problem itself can be viewed as a generalization of the \emph{matrix scaling problem} for which Sinkhorn's algorithm provides a simple, iterative solution. 
A convergence analysis of Gurvits' algorithm was later presented by Garg et al.~\citep{garg,garg2018algorithmic}.  They give a 
$O \big(	{\rm poly} \big( b,d, \frac{1}{\epsilon}	\big) \big)$ guarantee, where $b$ denotes the bit complexity of the Brascamp--Lieb datum, which can be exponential in $m,d$~\citep[Theorem~1.5]{garg2018algorithmic}.  
A second-order approach by Allen-Zhu et al.~\citep{allen-zhu} for operator scaling achieves ${\rm poly}(\log \frac{1}{\epsilon})$~\citep[Theorem M1]{allen-zhu}, further reducing the dependency on $\frac{1}{\epsilon}$, if applied for the computation of Brascamp-Lieb constants. 
Furthermore,  under the additional assumption that $\Ac$ fulfills the \emph{$\lambda$-spectral gap condition} ($\sigma_2$ is the second singular value), i.e.,
\begin{align*}
\sigma_2 \Big(	\nlsum_{j \in [m]} A_j \otimes A_j	\Big) \leq (1 - \lambda) \frac{\sum_{j \in [m]} \norm{A_j}_{\rm F}^2}{\sqrt{d d'}} \; ,
\end{align*}
stronger results can be shown (here, $d_j=d'$ for all $j \in [m]$). 
Indeed, Kwok et al.~\citep{kwok} show that
a gradient-based method achieves a factor $\log \frac{1}{\epsilon}$~\citep[Theorem~1.5]{kwok}. Using second-order methods, B{\"u}rgisser et al.~\citep{burgisser2019towards} give a $O(\log \frac{1}{\epsilon})$ complexity in the general cases, albeit still with an exponential dependence on the bit complexity of the weight vector $\vw$. This result also applies to the more general operator scaling problem.

The formulation of the Brascamp--Lieb problem utilized in these approaches differs from ours~\eqref{eq:F}. Notably, their formulation has an \emph{exponential} dependency on the bit complexity of the Brascamp--Lieb datum.  In contrast, ~\eqref{eq:F} is only \emph{polynomial} in the bit complexity of the datum~\citep{vishnoi}.  In addition,  while these approaches employ a more commonly used Riemannian structure on $\pd_d$ induced by the Frobenius norm,  our analysis uses the Thompson metric on $\pd_d$.

\section{Proofs of the results}\label{sec:proofs}
\subsection{Convergence of the fixed-point iteration using the map $G_t$}
\label{sec:convergence-G}
We begin by showing that the sequence of iterates $(X_k)_{k \in \mathbb{N}}$ produced by the map $G_t$ converges to a solution of Problem~\eqref{eq:F}:
\begin{theorem}[Theorem~\ref{thm:conv-G1}]
Let $X_0 \in D_{\delta, \Delta}$; the sequence $(X_k)_{k \in \mathbb{N}}$ generated by iterating $G_t$ converges to a solution of problem~\eqref{eq:F} as $k\to \infty$.
\end{theorem}
\noindent We first show the existence of a fixed point of the map $G$:
\begin{lem}\label{lem:fix-point}
$G$ has a fixed point in $D_{\delta, \Delta}$.
\end{lem}
\begin{proof}
From Theorem~\ref{thm:maximizer} we know that the minimum is attained in problem~\eqref{eq:F}, i.e., there exists a point $X^*$, such that $F(X^*) = \inf_{X} F(X)$. $X^*$ must be a stationary point, i.e., $\nabla F(X^*)=0$. Hence, by construction $X^* = G(X^*)$. Moreover,~\citep[Prop. 5.2]{tao-paper} implies that the minimum lies in a set of the form $D_{\delta,\Delta}$ (see section~\ref{sec:delta} for more details). This implies that $G$ has a fixed point in $D_{\delta,\Delta}$.
\end{proof}

\noindent For proving Theorem~\ref{thm:conv-G1}  we need a few additional auxiliary results that we now present. We start by showing that $G$ is homogeneous and order-preserving:
\begin{lem}\label{lem:G1-prop}
Let $G$ be as defined in~\eqref{eq:map}. Then:
\begin{enumerate}
\item For a scalar $\alpha \in \mathbb{R}$ we have $G(\alpha X)=\alpha G(X)$, i.e., $G$ is \emph{homogeneous}.
\item For $X,Y \in \mathbb{P}_d$ with $X \preceq Y$ ($X \prec Y$) we have $G(X) \preceq G(Y)$ ($G(X) \prec G(Y)$), i.e., $G$ is \emph{order-preserving}.
\end{enumerate}
\end{lem}
\begin{proof}
Property (1) is trivial. We therefore move directly to property (2). First, note that for $X \preceq Y$, we have $A_j^\ast X A_j \preceq A_j^\ast Y A_j$.  Assuming that $A_j^\ast X A_j$ is invertible, this implies
\begin{equation*}
(A_j^\ast X A_j)^{-1} \succeq (A_j^\ast Y A_j)^{-1} \; .
\end{equation*}
From this, the claim follows immediately as
\begin{equation*}
   \Big( \nlsum_{j \in [m]} w_j A_j (A_j^\ast X A_j)^{-1} A_j^\ast \Big)^{-1} \preceq \Big( \nlsum_{j \in [m]} w_j A_j (A_j^\ast Y A_j)^{-1} A_j^\ast \Big)^{-1} \; .
\end{equation*}
The same arguments apply in the $\prec$ case.
\end{proof}
\noindent A direct consequence of this property is that $G$ is non-expansive.
\begin{lem}\label{lem:nonexp}
  $G$ is non-expansive.
\end{lem}
\begin{proof}
We must prove that $\delta_{\rm T}(G(X), G(Y)) \leq \delta_{\rm T}(X,Y)$. By Definition~\ref{def:thompson}(2) we have
$M(Y/X)^{-1} Y \preceq X \preceq M(X/Y) Y$, and also
\begin{equation} \label{eq:G-bounds}
M(G(Y)/G(X))^{-1} G(Y) \preceq G(X) \preceq M(G(X)/G(Y)) G(Y) \; .
\end{equation}
Furthermore,  since $G$ is order-preserving and homogeneous (Lemma~\ref{lem:G1-prop}), we have
\begin{align*}
G(X) &\preceq G(M(X/Y) Y) = M(X/Y) G(Y) \\
G(X) &\succeq G(M(Y/X)^{-1} Y) = M(Y/X)^{-1} G(Y) \; .
\end{align*}
Since $M(G(Y)/G(X)), M(G(X)/G(Y))$ are extremal in inequality~\eqref{eq:G-bounds}, this implies
\begin{align*}
M(X/Y) &\geq M(G(X)/G(Y)) \\
M(Y/X) &\geq M(G(Y)/G(X)) \; ,
\end{align*}
which gives
\begin{align*}
\max \lbrace M(X/Y), M(Y/X) \rbrace \geq \max \lbrace M(G(X)/G(Y)), M(G(Y)/G(X)) \rbrace 
\end{align*}
and thereby the claim. 
\end{proof}
\noindent A direct consequence of the non-expansivity of $G$ is the following characterization of its ``average map'' $G_t$ (see, e.g.,~\citep{gornicki2019remarks}):
\begin{cor}
    $G_t$ is non-expansive and its fixed point set is that of $G$.
\end{cor}

\noindent Next, we establish the asymptotic regularity of $G_t$: Recall the following definition:
\begin{defn}[Asymptotic regularity]\label{def:asymp-reg}
    Let $T: \Xc \rightarrow \Xc$ ($\Xc \subset \pd_d$) denote a nonlinear map on the positive definite matrices. We say that $T$ is \emph{asymptotically regular} if $\lim_{k \rightarrow \infty} \delta_{\rm T}(T(X_{k+1}),T(X_k))=0$ for every sequence $(X_k)_{k \in \mathbb{N}}$, such that $X_0 \in \Xc$ and $X_{k+1}=T(X_k)$ for all $k \in \mathbb{N}$.
\end{defn}
\begin{lem}
$G_t$ is asymptotically regular with respect to the Thompson metric, i.e., 
\begin{equation*}
    \delta_{\rm T} \big(G_t(X_{k+1}), G_t(X_k) \big) \rightarrow 0 \; .
\end{equation*}
\end{lem}
\begin{proof}
    The result is a direct consequence of Ishikawa's theorem~\citep{ishikawa1976fixed}.
\end{proof}

\noindent For future reference we further note that the non-expansivity of $G$ implies boundedness:
\begin{lem}\label{lem:bounded}
$G$ is bounded on $D_{\delta, \Delta}$, i.e., there exists a $R>0$ depending on the input datum $(\Ac,\vw)$ only, such that $\norm{G(X_k)} < R$ for all $k \geq 0$.
\end{lem}
\begin{proof}
From Lemma~\ref{lem:fix-point} we know that ${\rm Fix(G) \neq \emptyset}$. Let $X^* \in {\rm Fix}(G)$. Due to the non-expansivity of $G$ (Lemma~\ref{lem:nonexp}), we have Fejér-monotonicity at $X^*$:
\begin{equation*}
    \delta_{\rm T} \big( X_{k+1}, X^* \big) = \delta_{\rm T} \big( G(X_k), G(X^*) \big) \leq \delta_{\rm T} \big( X_k, X^* \big) \; .
\end{equation*}
Thus, the iterates $\{G(X_k)\}_{k \geq 0}$ must remain bounded.
\end{proof}
\color{black}
\color{black}
\begin{rmk}\normalfont
We note that it is a well-known fact that non-expansive maps either do not have a fixed point or are bounded. Classical versions of this result can be found in~\citep[Corr. 6]{pazy_asymptotic_1971} (for Hilbert spaces) and in~\citep{reich_asymptotic_1973} (for Banach spaces).
\end{rmk}

\noindent In summary, we obtain our asymptotic convergence result (Theorem~\ref{thm:conv-G1}) as a direct consequence of Theorem~\ref{thm:s-h}:
With initialization at any $X \in D_{\delta, \Delta}$ the sequence produced by the map
$G_t$ converges to a solution of problem~\eqref{eq:F}. \\

\noindent Finally, we want to give a lower bound on the sequence of iterates produced by the map $G$, which is needed for the subsequent analysis. 

We show the following bound:
\begin{lem}\label{lem:G-low-bound}
Every sequence $(X_k)_{k \in \mathbb{N}}$ with $X_0 \in D_{\delta,\Delta}$ and $X_{k+1}=G(X_k)$ for all $k \in \mathbb{N}$,
is bounded from below, i.e., $G(X_k) \succeq \delta I$.
\end{lem}
\begin{proof} 
Let $X \in D_{\delta,\Delta}$ be arbitrary; we want to show that $G(X) \succeq \delta I$. 
From Lemma~\ref{lem:G1-prop}(2) we know that $G$ is order-preserving, i.e. $X \succeq\delta I$ implies $G(X) \succeq G(\delta I)$.
Therefore, it remains to be shown that $G(\delta I) \succeq \delta I$. 
For this, notice that 
\begin{equation*}
    A(A^* Z A)^{-1} A^* \preceq Z^{-1} \Leftrightarrow \begin{pmatrix}
        Z^{-1} & A \\
        A^* & A^* Z A 
    \end{pmatrix} \succeq 0 \; ,
\end{equation*}
where the right inequality follows from the following factorization via Schur complements:
\begin{equation*}
    \begin{pmatrix}
        Z^{-1} & A \\
        A^* & A^* Z A
    \end{pmatrix}
    = \begin{pmatrix}
        I & 0 \\
        0 & A^*
    \end{pmatrix} \begin{pmatrix}
        Z^{-1} & I \\
        I & Z
    \end{pmatrix} \begin{pmatrix}
        I & 0 \\
        0 & P
    \end{pmatrix} \; .
\end{equation*}
\noindent This lemma implies in particular that 
\begin{equation*}
A(A^* Z A)^{-1} A^* \preceq Z^{-1} \; .
\end{equation*}
For $Z = \delta I$,  this gives (together with the min-max theorem)
\begin{align*}
\lambda_{\max}\bigl(\big(G(\delta I)\big)^{-1}\bigr) = \lambda_{\max} \bigl(	A_j^* \bigl( A_j^* (\delta I) A_j	\bigr)^{-1} A_j	\bigr) \leq \lambda_{\max} \bigl(	\delta^{-1} I \bigr) \; .
\end{align*}
With $\sum_{j \in [m]} w_j =1$, this implies $\lambda_{\max}\bigl(\bigl(G(\delta I)\bigr)^{-1}\bigr) \leq \delta^{-1}$, i.e. $(G(\delta I))^{-1} \preceq \delta^{-1} I$. This gives $G(\delta I) \succeq \delta I$. 
\end{proof}

\subsection{Non-asymptotic convergence analysis of $G_\mu$}
\label{sec:convergence-G-reg}
We start by showing that the regularized map $G_\mu$ has the crucial contractivity property.
\begin{lem}[Lemma~\ref{lem:G4-contract}]
$G_\mu$ is contractive.
\end{lem}
\begin{proof}
\begin{align*}
\delta_{\rm T} (G_\mu(X), G_\mu(Y)) &= \delta_{\rm T} \big( \mu I + G^{-1}(X), \mu I + G^{-1}(Y)	\big) \\
&\overset{(1)}{\leq} \frac{\gamma}{\gamma + \mu} \delta_{\rm T} \big(  G^{-1}(X),  G^{-1}(Y)	\big) \\
&\overset{(2)}{=} \frac{\gamma}{\gamma + \mu} \delta_{\rm T} \big(  G(X),  G(Y)	\big) \\
&\overset{(3)}{\leq}\frac{\gamma}{\gamma + \mu} \delta_{\rm T} \big( X,  Y	\big) \; ,
\end{align*}
where ($1$) follows from Lemma~\ref{lem:prop-dT}(6) with $\gamma = \max \lbrace \norm{G^{-1}(X)}, \norm{G^{-1}(Y)}	\rbrace \leq \frac{1}{\sqrt{\delta}}$ by Lemma~\ref{lem:G-low-bound};  ($2$) follows from Lemma~\ref{lem:prop-dT}(1) and ($3$) from $G$ being log-nonexpansive (Lemma~\ref{lem:nonexp}).  Since $\mu>0$, this implies $\delta_{\rm T} (G_\mu(X), G_\mu(Y)) < \delta_{\rm T} (X,Y)$. 
\end{proof}
The contractivity of the regularized map furthermore allows us to quantify its convergence rate non-asymptotically. 
\begin{theorem}
For a suitable choice of $\mu$, $G_{\mu}$ computes an $\epsilon$-approximate solution to problem (Eq.~\eqref{eq:F}) with iteration complexity $O \Big( \log \big( \frac{1}{\epsilon},\frac{\Delta}{\delta^{3/2}} ,\sqrt{d}, \frac{1}{R}\big) \Big)$, where $d, \delta, \Delta, R$ depend on the input datum $(\Ac,\vw)$ only.
\end{theorem}
\begin{proof}
\noindent Our plan is as follows: We start by analyzing the choice of the parameter $\mu$ and the convergence rate of the regularized approach. This will allow us to develop a non-asymptotic convergence result for problem~\eqref{eq:F}. 

The optimality gap is given by
\begin{align*}
&\vert F(X_k) - F(X^\ast) \vert \\
&\qquad\leq \vert F(X_k) - F_\mu (X_k) \vert + \vert F_\mu (X_k) - F_\mu (X^*) \vert +  \vert F_\mu (X^*)- F(X^*) \vert \; ,
\end{align*}
where the first and third terms represent the \emph{approximation error} of $F_\mu$ and the second term the \emph{convergence rate} of $G_\mu$ in terms of the function values. Therefore, we have to perform the following steps:
\begin{enumerate}
\item bound the approximation error of $F_\mu$, i.e., bound $\vert F(X) - F_\mu (X) \vert$; 
\item characterize the convergence rate of $G_\mu$, i.e., bound $\vert F_\mu(X) - F_\mu (X^\ast) \vert$.
\end{enumerate}

\paragraph{Approximation error of $F_\mu$.}
By construction we have $F_\mu (X) - F(X) = \mu \trace (X)$. Note that at a fixed point $X$, we have
\begin{align*}
\trace (X) = \trace (G_\mu (X)) = \trace \Big( \big( \mu I + \underbrace{\nlsum_{j \in [m]} w_j A_j (A_j^\ast X A_j)^{-1} A_j^\ast	}_{=:G^{-1}(X)}	\big)^{-1} \Big) \; .
\end{align*}
Let $\lbrace \nu_1, \dots, \nu_n \rbrace$ denote the eigenvalues of $G^{-1}(X)$. Then the eigenvalues of $G^{-1}(X) + \mu I$ are $\lbrace \nu_1 + \mu, \dots, \nu_n + \mu \rbrace$, which gives (for a fixed point $X$)
\begin{equation*}
\trace X = \trace \big[ \big(	\mu I + G^{-1}(X)	\big)^{-1} \big] = \sum_{i \in [n]} \frac{1}{\lambda_i \big(	\mu I + G^{-1}(X)	\big)} = \sum_{i \in [n]} \frac{1}{\mu + \nu_i} \; .
\end{equation*}
Let $\nu_{\rm min}$ be the minimal eigenvalue of $G^{-1}(X)$. Then
\begin{equation*}
\trace X = \sum_{i \in [n]} \frac{1}{\mu + \nu_i} \leq \sum_{i \in [n]} \frac{1}{\mu + \nu_{\rm min}} = \frac{d}{\mu + \nu_{\rm min}} \; .
\end{equation*}
Due to the boundedness of $G$ (Lemma~\ref{lem:bounded}), we have that $\nu_{\min} \geq \frac{1}{R}$.
Inserting this above, we obtain the bound
\begin{equation*}
\trace (X) = \frac{d}{\mu + \nu_{\rm min}} \leq \frac{d}{\mu + \frac{1}{R}}  \; ,
\end{equation*}
which in turn implies the bound
\begin{equation}\label{eq:choose-mu}
F_\mu (X) - F(X) = \mu \trace (X) \leq \frac{\mu d}{\mu + \frac{1}{R}} \; .
\end{equation}
This bound gives a means for choosing $\mu$ ($d, \delta, R$ are determined by the input $(\Ac,\vw)$), such that a solution to the regularized problem gives an $\frac{\epsilon}{4}$-accurate solution to the original one. We defer the further discussion to the end of the proof.

\paragraph{Convergence rate of $F_\mu$.}
From Lemma~\ref{lem:G4-contract} we know that $G_\mu$ is log-contractive, i.e., we have
\begin{equation*}
\delta_{\rm T} (\underbrace{G_\mu(X_k)}_{=X_{k+1}},\underbrace{ G_\mu(X^*)}_{=X^*}) \leq \frac{\gamma}{\gamma + \mu} \delta_{\rm T} (X_k, X^*) \; ,
\end{equation*}
which implies that
\begin{equation*}
\delta_{\rm T}(X_{k+1}, X^\ast) \leq  \frac{\gamma}{\gamma + \mu} \delta_{\rm T} (X_k, X^*) \; ,
\end{equation*}
where $\gamma := \max \lbrace \norm{G^{-1}(X_k)}, \norm{G^{-1}(X^*)}		\rbrace$. Recursively, this gives
\begin{align}\label{eq:G4-bound}
\delta_{\rm T} (X_k, X^*) \leq \Big(	\frac{\tilde{\gamma}}{\tilde{\gamma} + \mu}	\Big)^k \delta_{\rm T} (X_0, X^*) \; ,
\end{align}
where we set $\tilde{\gamma} := \max \lbrace \max_{1 \leq i \leq k} \norm{G^{-1}(X_i)}, \norm{G^{-1}(X^*)} \rbrace$, which is bounded as $\frac{1}{R} \leq \tilde{\gamma} \leq \frac{1}{\sqrt{\delta}}$ by construction of the set $D_{\delta,\Delta}$ and the boundedness of $G$.
If $F_\mu$ is $L$-smooth, then the following bound holds:
\begin{equation*}
F_\mu (X^\ast) \leq F_\mu (X_k) + \ip{\nabla F_\mu (X_k)}{X_k - X^\ast}_{\rm F} + \frac{L}{2} \norm{X^* - X_k}_{\rm F}^2 \; .
\end{equation*}
This implies that
\begin{equation*}
    \left\vert F_\mu (X^\ast) - F_\mu (X_k) \right\vert \leq \left\vert  \ip{\nabla F_\mu (X_k)}{X_k - X^\ast}_{\rm F} \right\vert + \frac{L}{2} \norm{X^* - X_k}_{\rm F}^2 \; ,
\end{equation*}
where$\ip{\cdot}{\cdot}_{\rm F}$ denotes the usual Frobenius product $\ip{X}{Y}_{\rm F} = \trace \big( X^* Y \big)$. In order to use this bound in our analysis, we must show that $F_\mu$ is indeed $L$-smooth. To that end, recall that
\begin{align*}
\nabla F_\mu(X) = \nlsum_{j \in [m]}  w_j A_j (A_j^* X A_j)^{-1} A_j^* - X^{-1}+ \mu I \; 
\end{align*}
and let $H_1(X) := \sum_{j \in [m]} w_j A_j \big( A_j^* X A_j\big)^{-1} A_j^* $ and $H_2(X) := - X^{-1}$.
Then $${\rm Lip}(\nabla F_\mu) \leq L$$ if ${\rm Lip}(H_1) + {\rm Lip}(H_2) \leq L$ for some $L$.  For ${\rm Lip}(H_1)$ note that
\begin{align*}
&\nabla \Big( 	\nlsum_{j \in [m]} w_j A_j (A_j^* X A_j)^{-1} A_j^*	\Big) \\
&\qquad= - \nlsum_{j \in [m]} w_j \big[ \big(A_j (A_j^* X A_j)^{-1} A_j^* \big) \otimes \big(A_j (A_j^* X A_j)^{-1} A_j^* \big) \big] \\
&\qquad= - \Big(	\nlsum_{j \in [m]} w_j A_j\big(A_j^* X A_j \big)^{-1} A_j^* \Big) \otimes \Big(	\nlsum_{j \in [m]} w_j A_j\big(A_j^* X A_j \big)^{-1} A_j^* \Big)\\
&\qquad= - G^{-1}(X) \otimes G^{-1}(X) \; ,
\end{align*}
where the second equality follows from the distributive law. Then the construction of $D_{\delta,\Delta}$ and the fact that $G$ is a self-map imply
\begin{align*}
\norm{\nabla H_1(X)} = \norm{G^{-1}(X)}^2 \leq \frac{1}{\delta} \; .
\end{align*}
Furthermore, we have
\begin{align*}
\nabla \big(-X^{-1} \big) = (X^{-1}) \otimes (X^{-1}) \; ,
\end{align*}
which implies that $\norm{\nabla H_2(X)} = \norm{X^{-1}}^2 \leq \frac{1}{\delta}$, again by contruction of $D_{\delta,\Delta}$.
This implies $L=\frac{2}{\delta}$ and therefore Lemma~\ref{lem:F-smooth}.  

With that, the Lipschitz bound above implies
\begin{align*}
\vert F_\mu (X^\ast) - F_\mu (X_k) \vert &\leq \vert  \ip{\nabla F_\mu (X_k)}{X_k - X^\ast}_{\rm F} \vert + \frac{1}{\delta} \norm{X^* - X_k}_{\rm F}^2  \; .
\end{align*}
Hence, for the first term on the right hand side, the
Cauchy--Schwarz inequality implies
\begin{align*}
\vert \ip{\nabla F_\mu (X_k)}{X_k - X^\ast}_{\rm F} \vert &\leq \norm{\nabla F_\mu (X_k)}_{\rm F} \norm{X^* - X_k}_{\rm F} \\
&\leq \sqrt{d} \norm{\nabla F_\mu (X_k)} \norm{X^* - X_k}_{\rm F} \; .
\end{align*}
It remains to derive bounds on (a) $\norm{\nabla F_\mu (X_k)}$ and (b) the convergence rate of $( \norm{X^* - X_k}_{\rm F} )_{k\in \mathbb{N}}$. For (a), we have
\begin{align*}
\norm{\nabla F_\mu (X)} 
&= \norm{\nabla \big( F(X) + \mu \trace(X)		\big)} \\
&= \norm{ \underbrace{\big(\sum_{j \in [m]} w_j A_j (A_j^* X A_j)^{-1} A_j	\big)}_{=G^{-1}(X)} - X^{-1} + \mu I} \; .
\end{align*}
Recall that for $X,Y$ symmetric, Weyl's Perturbation Theorem (see, e.g.,~\citep[ III.2.7]{bhatia}) gives
\begin{equation*}
\norm{X - Y} \leq \max_i \left\vert \lambda_i (X) - \lambda_{n-i+1} (Y) \right\vert \; .
\end{equation*}
With that, we have
\begin{align*}
 \norm{ (	G^{-1}(X) + \mu I	) - X^{-1} }
 &\leq \max_i \vert \lambda_i ( G^{-1}(X) + \mu I) - \lambda_{n-i+1}  (X^{-1}) \vert \\
 &\leq \max \big\{ \lambda_{\max} ( G^{-1}(X) + \mu I	) , \lambda_{\max}  (X^{-1})	\big\} \; .
\end{align*}
Note that $G(X), X \succeq \delta I$ by Lemma~\ref{lem:G-low-bound} and the construction of the set $D_{\delta,\Delta}$. This implies $\lambda_{\min}(G(X)), \lambda_{\min}(X) \geq \delta$ and therefore $$\lambda_{\max}(G^{-1}(X)), \lambda_{\max}(X^{-1}) \leq \frac{1}{\delta} \; .$$ Inserting this above, we get
\begin{align}\label{eq:bound-grad-F}
 \norm{\nabla F_\mu (X)}  
&\leq  \max \big\{  \lambda_{\max} \big( G^{-1}(X) + \mu I	\big) , \lambda_{\max}  (X^{-1})	\big\} \\
&\leq \max \Big\{ \frac{1}{\delta} + \mu, \frac{1}{\delta}	\Big\} =\Big(\frac{1}{\delta} + \mu \Big) =: C_1 \; .
\end{align}
\noindent For (b), quantifying the convergence rate of $\big( \norm{X^* - X_k}_{\rm F} \big)_{k \in \mathbb{N}}$, we
recall an inequality of Snyder that relates the Thompson metric to Schatten norms.
\begin{theorem}[\citep{snyder}, Thm. 3]
Let $X,Y \in \pd_d$ and $\norm{\cdot}_p$ denotes the Schatten $p$-norm. Then:
\begin{align*}
\norm{X-Y}_p &\leq 2^{\frac{1}{p}} \frac{e^{\delta_{\rm T}(X,Y)} - 1}{e^{\delta_{\rm T}(X,Y)}} \max \big\{ \norm{X}_p, \norm{Y}_p \big\} \; .
\end{align*}
\end{theorem}
\noindent Since $\norm{\cdot}_{\rm F}$ is Schatten $2$-norm, we therefore get 
\begin{align}\label{eq:snyder-bound}
\norm{X_k-X^*}_{\rm F} 
&\leq \sqrt{2} \max \big\{ \norm{X_k}_{\rm F}, \norm{X^*}_{\rm F} \big\} \frac{e^{\delta_{\rm T}(X_k,X^*)} - 1}{e^{\delta_{\rm T}(X_k,X^*)}} \\
&\le  \sqrt{2d} \max \big\{ \norm{X_k}, \norm{X^*} \big\} \frac{e^{\delta_{\rm T}(X_k,X^*)} - 1}{e^{\delta_{\rm T}(X_k,X^*)}} \; .
\end{align}
For ease of notation, we introduce the shorthand $\delta_{T,k}:=\delta_{\rm T}(X_k, X^*)$. We have seen above that $G_\mu$ is strictly contractive and that in particular $\delta_{T,k} \leq \big( \frac{\tilde{\gamma}}{\tilde{\gamma} + \mu}		\big)^k \delta_{T,0}$.
If we initialize $X_0 \in D_{\delta,\Delta}$, then the initial distance from the optimum must be bounded due to the boundedness of $G$. Hence, there is some constant $C_2>0$, such that $\delta_{T,0} \leq C_2$. With this, we can analyze the exponential terms in~\eqref{eq:snyder-bound}. From the above inequality and the monotonicity of the exponential function, we have $e^{\delta_{T,k}} \leq e^{\big( \frac{\tilde{\gamma}}{\tilde{\gamma} + \mu}	\big)^k \delta_{T,0}}
= e^{C_2 \alpha^k}$, where we set $\Big( \frac{\tilde{\gamma}}{\tilde{\gamma} + \mu}		\Big)^k =: \alpha^k$.
Inserting this into the Lipschitz bound above, we then obtain
\begin{align*}
\left\vert F_\mu (X^\ast) - F_\mu (X_k) \right\vert 
&\leq \sqrt{d}\big\|\nabla F_\mu (X_k) \big\| \big\|X^* - X_k\big\|_{\rm F} + \frac{1}{\delta} \big\|X^* - X_k\big\|_{\rm F}^2 \\
&\overset{(1)}{\leq} \Big(\sqrt{2} C_1 + \frac{1}{\delta} \Big) \big\|X^* - X_k \big\|_{\rm F} \\
&\overset{(2)}{\leq} \underbrace{ \Big(\sqrt{2} C_1 + \frac{1}{\delta} \Big) \sqrt{\frac{2d}{\delta}} }_{=: \widehat{C}_1} \frac{e^{\delta_{T,k}} -1}{e^{\delta_{T,k}} }\; ,
\end{align*}
where (1) follows since $\Big\{ \Big(\frac{e^{\delta_{T,k}} -1}{e^{\delta_{T,k}} }\Big)^2 \Big\}_k$  decays faster than $\Big\{ \frac{e^{\delta_{T,k}} -1}{e^{\delta_{T,k}} } \Big\}_k$ and (2) follows from Snyder's theorem, and upon noting that
\begin{equation*}
    \max \Big\{ \max_{1 \leq i \leq k} \big\|X_i^{-1} \big\|, \big\|\big(X^*\big)^{-1}\big\| \Big\} \leq \frac{1}{\sqrt{\delta}},
\end{equation*}
by construction of $D_{\delta,\Delta}$.  The constant $\widehat{C}_1$ summarizes the dependency on the input datum $(\Ac,\vw)$ and the regularization parameter $\mu$.

We want to bound the iteration complexity of an $\epsilon'$-accurate solution ($\epsilon' := \frac{\epsilon}{2}$) of $F_\mu$, so let $\epsilon'=\left\vert F_\mu (X^\ast) - F_\mu (X_k) \right\vert$. Then with the above, we have
\begin{equation*}
\epsilon' \leq \widehat{C}_1  \frac{e^{\delta_{T,k}} - 1}{e^{\delta_{T,k}}} \;.
\end{equation*}
Furthermore, with the constants $\alpha$ and $C_2$ as introduced above, we get
\begin{equation*}
\epsilon' \leq \widehat{C}_1 \frac{e^{C_2 \alpha^k} - 1}{e^{C_2 \alpha^k}} = \widehat{C}_1 \Big(	1 - \frac{1}{e^{C_2 \alpha^k}}		\Big)\;. 
\end{equation*}
This can be rewritten as follows:  
\begin{align*}
\frac{\epsilon'}{\widehat{C}_1} \leq  1 - \frac{1}{e^{C_2 \alpha^k}} 
&\Leftrightarrow e^{C_2 \alpha^k} \geq \frac{1}{1 - \frac{\epsilon'}{\widehat{C}_1}} 
 \Leftrightarrow \alpha^k \geq \frac{1}{C_2} \log \Bigg( \frac{1}{1 - \frac{\epsilon'}{\widehat{C}_1}}	\Bigg) \\
&\Leftrightarrow \Bigg(	\frac{1}{\alpha}	\Bigg)^{-k} \geq \frac{1}{C_2} \log \Bigg( \frac{1}{1 - \frac{\epsilon'}{\widehat{C}_1}}	\Bigg) \\
&\Leftrightarrow -k \geq \frac{\log \Bigg(	\frac{1}{C_2} \log \Big( \frac{1}{1 - \frac{\epsilon'}{\widehat{C}_1}}	\Big)	\Bigg)}{\log \Big( \frac{1}{\alpha}	\Big)} \; .
\end{align*}
In summary, we have the following bound on the number of iterations needed to achieve an $\epsilon$-accurate solution.
\begin{equation}
k \leq -\frac{\log \Bigg(	\frac{1}{C_2} \log \Big( \frac{1}{1 - \frac{\epsilon'}{\widehat{C}_1}}	\Big)	\Bigg)}{\log \Big( \frac{1}{\alpha}	\Big)} \; .
\end{equation}
We first evaluate the dependency on $\frac{1}{\epsilon}$.  Assuming $\epsilon$ is small,  we get for the inner term
\begin{equation*}
    - \log \Big( 1 - \frac{\epsilon'}{\widehat{C}_1} \Big) \geq \frac{\epsilon'}{\widehat{C}_1} \; .
\end{equation*}
This implies a $\log \left(\frac{1}{\epsilon} \right)$ dependency as
\begin{equation}
\boxed{
k \leq \frac{\log \left(	\frac{\widehat{C}_1 C_2}{\epsilon'}	\right)}{\log \left( \frac{1}{\alpha}	\right)}
}
\end{equation}

\paragraph{Constants.} The result above depends on three constants $\widehat{C}_1, C_2,\alpha$, which we relate to the complexity of the input $(\Ac,\vw)$, characterized by the constants $d, \delta, \Delta, R$. 

\begin{rmk}\normalfont
As stated earlier, those constants are not necessary mutually independent. In particular, we expect that $R$ may dependent on $\delta$, i.e., a more precise notation would read $R(\delta)$. However, to keep notation simply, we just write $R$. For feasible BL data, we expect $\delta$ to be small and $R \geq 1$.
\end{rmk}

For $\widehat{C}_1$, note that
\begin{align*}
\widehat{C}_1 &= \Big(\sqrt{2} C_1 + \frac{1}{\delta} \Big) \sqrt{\frac{2d}{\delta}} 
\leq 2 \sqrt{\frac{d}{\delta}} \Big( \frac{1}{\delta} + \mu \Big) + \frac{\sqrt{2d}}{\delta^{3/2}} \overset{(1)}{=} O \Big( \sqrt{d}, \frac{1}{\delta^{3/2}},\frac{1}{R}	\Big)
 \; .
\end{align*}
To see (1), recall that to control the approximation error (Eq.~\eqref{eq:choose-mu}), we set $\mu:= \frac{\epsilon' }{2R(d-\epsilon'/2)}$. \footnote{We note that the theoretical arguments presented in this proof apply to any $\mu$, which is small enough in the sense of Eq.~\eqref{eq:choose-mu}. In practise, it can be difficult to determine the parameter $R$, in which case an appropriate $\mu$ needs to be found via grid search.} 
Since $\widehat{C}_1$ enters logarithmically, this introduces a $O \Big( \log \big( \frac{1}{\delta^{3/2}}, \sqrt{d}  ,\frac{1}{R}\big) \Big)$ dependency. 
Furthermore, we have that 
\begin{equation*}
C_2 = \delta_{\rm T} (X_0, X^*) \leq {\rm diam}(D_{\delta,\Delta}) \leq \log \frac{\Delta}{\delta} \; ,
\end{equation*}
since by construction
\begin{equation*}
    {\rm diam}(D_{\delta,\Delta}) = \max_{X,Y \in D_{\delta,\Delta}} \delta_{\rm T} (X,Y) = \log M \big( \Delta I_d / \delta I_d \big) = \log \frac{\Delta}{\delta} \; ,
\end{equation*}
which only depends on the parameters $\delta, \Delta$. 
Finally, we have
\begin{align*}
 \frac{1}{\alpha} =\frac{\tilde{\gamma} + \mu}{\tilde{\gamma}}  \geq \frac{\frac{1}{R} + \mu}{\frac{1}{\sqrt{\delta}}} 
 \geq \sqrt{\delta}\Big( \frac{1}{R} + \frac{\epsilon'}{2Rd} \Big) =  O \Big(\frac{\sqrt{\delta}}{R} \Big) \; .
\end{align*}
This implies that the contribution of the factor $\big(\log \big( \frac{1}{\alpha}\big)\big)^{-1}$ is dominated by the other terms. 

Overall, the regularized approach has a complexity of $O \Big( \log \big( \frac{1}{\epsilon},\frac{\Delta}{\delta^{3/2}} ,\sqrt{d}, \frac{1}{R}\big) \Big)$ with respect to the accuracy $\epsilon$ and $d, \delta,\Delta,R$, which depend on the input datum $(\Ac,\vw)$ only.  We will discuss the dependency on the parameter $\delta,\Delta$ in more detail in the following section.
\end{proof}
%

\subsection{Dependency on the input datum}
\label{sec:delta}
Our convergence guarantees rely on the existence of suitable parameters $\Delta, \delta>0$ for a given Brascamp--Lieb datum $(\Ac,\vw)$, such that the global optimum $X^*$ of Eq.~\eqref{eq:F} is contained in a set of the form $D_{\delta,\Delta}$.  The existence of such a $\delta$ follows from a result in~\citep{tao-paper}, which ensures that the $X^*$ is bounded from above and below. In this section, we recall the argument by~\citep{tao-paper} and highlight its connection to our parameters $\delta,\Delta$. Furthermore, we discuss a conjectured connection to the \emph{subdeterminant maximization} problem. The development of an efficient approach for computing $\delta$ (and hence determining $D_{\delta,\Delta}$ explicitly) is left for future work.\\

\noindent We recall some crucial properties of the optimum $X^*$ from~\citep{tao-paper}. The existence of suitable upper and lower bounds on $X^*$ is implied by~\citep[Proposition~5.2]{tao-paper}, which relies on the following lemma:
\begin{lem}[\citep{tao-paper}, Lemma~5.1]\label{prop:5.1}
Assume that $(\Ac,\vw)$ is feasible, i.e., the conditions in Theorem~\ref{thm:feasible} hold. Then there exist a $c>0$ and an orthonormal basis $\lbrace e_1, \dots, e_d \rbrace$ of $\mathbb{R}^d$ such that there exists an index set $I_j$ with $\vert I_j \vert = \dim(H_j) = \dim (A_j H)$, such that $\sum_{j \in [m]} w_j \vert I_j \cap \lbrace 1, \dots, k \rbrace \vert \leq k \quad \forall \; 0 \leq k \leq d$
and 
\begin{equation}\label{eq:36}
\Big\| \bigwedge_{i \in I_j} A_j e_i \Big\| \geq c \quad \forall \; 1 \leq j \leq m \; .
\end{equation}
Furthermore, 
\begin{equation}\label{eq:37}
\sum_{j \in [m]} w_j \left\vert I_j \cap \lbrace k+1, \dots, d \rbrace \right\vert \leq d-k \quad \forall \; 0 \leq k \leq d \; .
\end{equation}
\end{lem}
The connection to our parameters $\delta,\Delta$ follows from a crucial property, which is implicitly established in the proof of~\citep[Proposition~5.2]{tao-paper}:
\begin{cor}\label{cor:delta}
Let $\big(	\Ac, \vw \big)$ denote a simple input datum.
Then the optimizer $X^*$ for problem~\eqref{eq:F} is bounded from above and below, with both bounds depending on the input datum only.
\end{cor}
In short, the argument is as follows: We again consider Gaussian inputs $(Z_j)_{j \in [m]}$ and set $\sum_{j \in [m]} w_j \big(	A_j^* Z_j A_j 	\big) =: X$.\footnote{Note that here, again, we perform a variable change as in Theorem~\ref{thm:maximizer} to be consistent with our geodesically convex formulation (problem~\eqref{eq:F}).} By Lemma~\ref{prop:5.1},  this transformation is self-adjoint, since the matrices $\big( A_j \big)_{j \in [m]}$ are non-degenerate and $w_j >0$.   Therefore, we can diagonalize $X= P^{-1} {\rm diag}\big(\lambda_1, \dots, \lambda_d \big) P$, where $\lambda_d >0$.  As a consequence, there must exists a $\delta>0$, such that $\lambda_d \geq \delta$. 

Furthermore, from Eq.~\eqref{eq:36}, we see that $\lbrace A_j e_i \rbrace_{i \in I_j}$ forms a basis of $H_j$ with a lower bound on degeneracy. Therefore
\begin{equation*}
    \det (Z_j) \leq C \prod_{i \in I_j} \lambda_i
\end{equation*}
for some $C>0$, which depends only on the input datum. This further implies
\begin{align*}
    \prod_{j \in [m]} \big( \det Z_j \big)^{w_j} 
    &\leq C \prod_{i \in [d]} \lambda_i^{\sum_{j \in [m]} w_j \left\vert I_j \cap \lbrace i \rbrace \right\vert} \\
    &\overset{(1)}{\leq} C \lambda_1^d \prod_{0 \leq k \leq d-1} \Big( \frac{\lambda_{k+1}}{\lambda_k} \Big)^{\sum_{j \in [m]} w_j \left\vert I_j \cap \lbrace k+1, \dots, d \rbrace \right\vert} \\
    &\overset{(2)}{\leq} C \lambda_1^d \prod_{0 \leq k \leq d-1} \Big( \frac{\lambda_{k+1}}{\lambda_k} \Big)^{d-k} \\
    &\overset{(3)}{\leq} C \lambda_1 \cdot \dots \cdot \lambda_k \; ,
\end{align*}
where (1) follows from telescoping and applying the scaling condition (Theorem~\ref{thm:feasible}(2)), (2) from Eq.~\eqref{eq:37} and (3) from reverse telescoping.
Since we have restricted ourselves to simple BL input data (see Theorem~\ref{thm:maximizer}), Eq.~\eqref{eq:37} holds with inequality (see also Definition~\ref{def:simple}). A refinement of the above argument then implies that
\begin{equation}\label{eq:38}
    \prod_{j \in [m]} \big( \det (Z_j) \big)^{w_j} \leq C \det(X) \prod_{1 \leq k \leq d-1} \Big( \frac{\lambda_{k+1}}{\lambda_k} \Big)^{c} \; ,
\end{equation}
for $c>0$ as in Lemma~\ref{prop:5.1}, which again depends only on the input datum. This implies that Eq.~\eqref{eq:BL-lieb} goes to zero as $\frac{\lambda_d}{\lambda_1}$ goes to zero. Hence, the supremum (Eq.~\eqref{eq:BL-sup}) must lie in a region with $\lambda_1 \leq C \lambda_d$. This implies upper and lower bounds on $X^*$ and thus the existence of a compact convex set of the form $D_{\delta,\Delta}$ that contains $X^*$.

To implement our proposed algorithm, we need to find a good initialization $X_0 \in D_{\delta,\Delta}$, which requires the explicit computation of $\delta$ and $\Delta$.  We leave this as an open problem.\\

\noindent We conjecture that there is a close relation between our $\delta$ and the parameter $c$ in Theorem~\ref{prop:5.1} (via Eq.~\eqref{eq:36}), which is characterized by the bit complexity of the input datum $(\Ac, w)$.  Notably,  $c$ is closely related to solving a constrained subdeterminant maximization problem over the set of linear transforms. This observation may serve as a starting point for future investigation into the computation of the parameter $\delta$. Given $\delta$ and $c$, the parameter $\Delta$ could then be inferred from Eq.~\eqref{eq:38}.

To see that $c$ is closely related to subdeterminant maximization, note that with an orthonormal basis and the index sets $I_j$ in Theorem~\ref{prop:5.1}, we have
\begin{align}
\ip{Z_j A_j e_i}{A_j e_i}_{H_j} = \ip{e_i}{A_j^* Z_j A_j e_i}_H \leq \frac{1}{w_j} \ip{e_i}{X e_i}_H = \frac{\lambda_i}{w_j} \; .
\end{align}
Furthermore, we have with Eq.~\eqref{eq:36} that $\big( A_j e_i	\big)_{i \in I_j}$ forms a basis of $H_j$ with a lower bound on degeneracy, characterized by $c$. 
Recall that $c$ is defined,  such that we can find for each $j \in [m]$ an index set $I_j \subseteq \lbrace 1, \dots, d \rbrace$ with $\vert I_j \vert = d_j$, such that 
\begin{equation*}
\Big\| \bigwedge_{i \in I_j} A_j e_i \Big\|_{d_j} > c \; .
\end{equation*}
Ignoring for now the complexity of finding such an index set,  we consider first the case where the $A_j$ are vectors $\big(	 a_1^j, \dots, a_d^j	\big)$, i.e. $\vert I_j \vert =1$. 
Then $\norm{a_j e_i} \geq \vert a_i^j \vert$
and therefore $c = \min_{j \in [m]} \vert a_i^j \vert = \min_{j \in [m]} \norm{a^j}_\infty$, i.e., each $A_j$ has an entry with absolute value at least $c$.  In the general case $A_j: \reals^d \rightarrow \reals^{d_j}$ we have (${\rm row}_k A$ indicating the $k$th row of $A$)
\begin{align*}
\Big\| \bigwedge_{i \in I_j} A_j e_i \Big\|_{d_j} &= \Big\|A_j e_{i_1} \wedge \dots \wedge A_j e_{i_{d_j}}\Big\|_{d_j} \\
&= \Big\| {\rm row}_{i_1} A_j \wedge \dots \wedge {\rm row}_{d_j} A_j \Big\|_{d_j}  \\
&= \Big\|	\det \big(	({\rm row}_{i_1} A_j)^T, \dots, ({\rm row}_{d_j} A_j)^T 	\big)	e_{i_1} \wedge \dots \wedge e_{i_{d_j}}	\Big\|_{d_j} \\
&= \left\vert 	\det \big(	({\rm row}_{i_1} A_j)^T, \dots, ({\rm row}_{d_j} A_j)^T 	\big) \right\vert \; .
\end{align*}
For each $j$, this translates to an optimization problem of the form
\begin{equation}\label{eq:subdet}
\max_{\vert I \vert = d_j} \left\vert \det \big( (A_j)_I \big) \right\vert \; ,
\end{equation}
which corresponds to a constrained \emph{subdeterminant maximization}.  For computing $c$, we simply take the minimum of all subdeterminant maximizers, i.e.
\begin{align*}
c := \min_{j \in [m]} \max_{\vert I \vert = d_j} \left\vert \det \big( (A_j)_I \big) \right\vert \; .
\end{align*}

Even the reduced problem~\eqref{eq:subdet}, which ignores the combinatorial constrained on the chosen index set $I_j$ is known to be NP-hard.  An approximation algorithm for the reduced problem, which achieves accuracy up to a factor of $O(e^{d'})$ was given by~\citep{nikolov}:
\begin{theorem}[\citep{nikolov}, Theorem~3]
There exists a deterministic polynomial time algorithm, which approximates problem~\eqref{eq:subdet} within a factor of $e^{d' + o(d')}$.
\end{theorem}
\noindent For a detailed analysis of the algorithmic aspects of this problem and Nikolov's approach, see also~\citep{vishnoi_subdet}.  An interesting direction for future work is the question of whether a suitable approximation algorithm for the parameter $c$ can be found, which, in turn, may provide insight into explicit constructions of the constants $\delta, \Delta$.

\section{Other Picard iterations for computing Brascamp--Lieb constants}\label{sec:other-picard}
This paper analyzes a specific Picard iteration, generated by the map $G$, for computing Brascamp--Lieb constants and analyzes its convergence via the Thompson part metric. We note that neither the choice of the Picard iteration, nor the specific Finslerian lens employed in the analysis is unique. Our particular choice is motivated by the observation that one can establish \emph{strict contractivity} of $G$ with respect to the Thompson part metric, which allows for developing \emph{non-asymptotic} convergence guarantees. In this section we discuss one other Picard iteration that arises from problem~\eqref{eq:F} and establish asymptotic convergence. 

Consider the map
\begin{equation}\label{eq:G-tilde}
    \tilde{G} := \frac{G}{\norm{G}} \; ,
\end{equation}
where $\norm{\cdot}$ may denote any matrix norm. We want to analyze $\tilde{G}$ with respect to the \emph{Hilbert projective metric}, which is given by
\begin{equation}
    \delta_{\rm H}(X,Y) := \log \big( M(X/Y) M(Y/X)	\big),
\end{equation}
where again $M(X/Y) := \inf \lbrace	\lambda > 0 : X \preceq \lambda Y	\rbrace$. We note the close relationship between the Thompson part metric $\delta_{\rm T}$ (see Definition~\ref{def:thompson}) and the Hilbert projective metric $\delta_{\rm H}$. Both metrics induce a Finslerian geometry on the manifold of positive definite matrices $\pd_d$.

One can show that $\tilde{G}$ is non-expansive and asymptotically regular with respect to $\delta_{\rm H}$ and that it has a (not necessary unique) fixed point in the compact set $\tilde{D}:= \{ X: \frac{X}{\norm{X}} \}$. Asymptotic convergence of $\tilde{G}$ is then guaranteed by a Hilbert geometry version of Theorem~\ref{thm:s-h}. Such a result can be obtained by specializing the more general Banach space version of ~\citep[Theorem 1.2]{baillon_asymptotic_1978} to Hilbert geometry.

We can show that the Picard iteration defined by the map $\tilde{G}$ converges to the same fixed-point as that of the map $G$: Let $X^*$ denote a fixed point of $\tilde{G}$. Then $G(X^*)=\lambda X^*$ for some $\lambda >0$. Eq.~\ref{eq:G-tilde} together with Corollary~\ref{cor:delta} imply that $\lambda=1$, which gives the claim.

It remains an open question to show, whether $\tilde{G}$ is strictly contractive with respect to the Hilbert projective norm. This would allow to adapt our non-asymptotic convergence analysis to the map $\tilde{G}$ as well.

\section{Discussion}\label{sec:discussion}
In this paper, we introduced a novel fixed-point approach for computing Brascamp--Lieb constants, which is grounded in nonlinear Perron--Frobenius theory.  In contrast to much of the prior literature, which has analyzed the problem through a Riemannian lens, our approach utilizes a Finslerian geometry on the manifold of symmetric, positive definite matrices, which arises from the Thompson part metric. To the best of our knowledge, this introduces a novel geometric perspective on the problem of computing Brascamp--Lieb constants for simple input data.

We establish convergence of our proposed fixed-point approach and present a full convergence analysis for a regularized variant. Importantly, we show that the regularized approach attains a non-asymptotic convergence rate of $O \Big( \log \big( \frac{1}{\epsilon},\frac{\Delta}{\delta^{3/2}} ,\sqrt{d}, \frac{1}{R}\big) \Big)$, where $d, \delta,\Delta,R$ depend on the size and structure of the input datum only. 
A shortcoming of the present approach lies in the difficulty of characterizing the structural dependency on the input datum more explicitly. In particular, while~\citep[Proposition 5.2]{tao-paper} guarantees the existence of suitable $\delta, \Delta>0$ for each (feasible, simple) Brascamp--Lieb datum, we were unable to develop an approach for explicitly computing their values. Access to suitable $\delta,\Delta$ would allow us to find a good initialization for our proposed algorithm and therefore render our approach into a practical method. In addition, we expect the parameter $R$ to depend on $\delta$, but were unable to derive a precise relation. The mutual dependency of the parameters $d, \delta, \Delta, R$ is left for future work.

In future work, we hope to further investigate approaches for computing $\delta, \Delta$ explicitly. In addition, we hope to apply the techniques presented here to geodesically convex optimization problems similar to ~\eqref{eq:F}, such as the operator scaling problem~\citep{garg2018algorithmic} and difference of convex functions~\citep{sra-weber-2022}.
Other avenues for future investigation include practical applications of our fixed-point approach.  Due to the strong connections of Brascamp--Lieb inequalities to important questions in machine learning and information theory, this approach could be of wider interest.

More generally, we hope that the Finslerian lens on fixed point iterations provides a new perspective on the problem of computing Brascamp--Lieb constants. We believe that the tools developed in this work can be applied to a wider class of Picard iterations that arise in the context of Brascamp--Lieb constants (such as those described in section~\ref{sec:other-picard}).

\subsection*{Acknowledgments}
This project was started during a visit of MW to MIT, supported by an Amazon Research Award. Part of this work was done while MW visited the Simons Institute for the Theory of Computing in Berkeley, CA, supported by a Simons-Berkeley Research Fellowship. SS acknowledges support from an NSF-CAREER award (1846088).\\

\noindent The authors thank Terence Tao, Brian Lins and two anonymous reviewers for helpful comments.

\setlength{\bibsep}{3pt}
\bibliographystyle{plainnat}
\bibliography{ref}


\end{document}